\documentclass[sn-mathphys,Numbered]{sn-jnl}% Math and Physical Sciences Reference Style
\usepackage{graphicx}%
\usepackage{multirow}%
\usepackage{amsmath,amssymb,amsfonts}%
\usepackage{amsthm}%
\usepackage{mathrsfs}%
\usepackage[title]{appendix}%
\usepackage{xcolor}%
\usepackage{textcomp}%
\usepackage{manyfoot}%
\usepackage{booktabs}%
\usepackage{algorithm}%
\usepackage{algorithmicx}%
\usepackage{algpseudocode}%
\usepackage{listings}%
\theoremstyle{thmstyleone}%
\newtheorem{theorem}{Theorem}%  meant for continuous numbers
\newtheorem{proposition}[theorem]{Proposition}%

\theoremstyle{thmstyletwo}%

\theoremstyle{thmstylethree}%
\newtheorem{definition}{Definition}%

\begin{document}

\title[M.A.P.R]{Souriau-Fisher metric and Completely integrable system on Lie groups $SO(2)$ and $SO(3)$.}

%%=============================================================%%
%% Prefix   -> \pfx{Dr}
%% GivenName    -> \fnm{Joergen W.}
%% Particle -> \spfx{van der} -> surname prefix
%% FamilyName   -> \sur{Ploeg}
%% Suffix   -> \sfx{IV}
%% NatureName   -> \tanm{Poet Laureate} -> Title after name
%% Degrees  -> \dgr{MSc, PhD}
%% \author*[1,2]{\pfx{Dr} \fnm{Joergen W.} \spfx{van der} \sur{Ploeg} \sfx{IV} \tanm{Poet Laureate}
%%                 \dgr{MSc, PhD}}\email{iauthor@gmail.com}
%%=============================================================%%
\author*[1]{\fnm{} \sur{Prosper Rosaire Mama  Assandje }}\email{mamarosaire@facsciences-uy1.cm}
\author[1]{\fnm{} \sur{Michel Bertrand Djiadeu Ngaha }}\email{michel.djiadeu@facsciences-uy1.cm}
\equalcont{These authors contributed equally to this work.}
\author[3]{\fnm{} \sur{Romain Nimpa Pefoukeu }}\email{romain.nimpa@facsciences-uy1.cm}
\equalcont{These authors contributed equally to this work.}
\author[4]{\fnm{} \sur{Salomon Joseph  Mbatakou }}\email{salomon-joseph.Mbatakou@facsciences-uy1.cm}
\equalcont{These authors contributed equally to this work.}

\affil*[1]{\orgdiv{Department of Mathematics}, \orgname{University
of Yaounde 1}, \orgaddress{\street{usrectorat@.univ-yaounde1.cm},
\city{yaounde}, \postcode{337}, \state{Center}, \country{Cameroon}}}

%\affil[2]{\orgdiv{Department of Mathematics and
%Computer Science}, \orgname{University of Maroua},
%\orgaddress{\street{www.fs.univ.maroua.cm}, \city{Maroua},
%\postcode{814}, \state{Far -North}, \country{Cameroon}}}
%\affil[3]{\orgdiv{Department of Mathematics and Computer Science},
%\orgname{University of Maroua},
%\orgaddress{\street{www.fs.univ.maroua.cm}, \city{Maroua},
%\postcode{814}, \state{Far -North}, \country{Cameroon}}}
%\affil[4]{\orgdiv{Department of Mathematics}, \orgname{Nagoya
%University},
%\orgaddress{\street{nakamura.takemi.d7@s.mail.nagoya-u.ac.jp},
%\city{Nagoya}, \postcode{464-8601}, \state{Nagoya},
%\country{Japan}}}

%%==================================%%
%% sample for unstructured abstract %%
%%==================================%%

\abstract{We study the generalize Fisher metric on $SO(2)$ and
$SO(3)$ via the thermodynamics Lie group theories of Souriau. Then
we give the effect of 2-cocycle on the integrability of gradient
systems due to  the Fisher metric and Souriau-Fisher metric. In
addition, we show how the cocycle can locally modify the Fisher
metric on a coadjoint orbit, in explicit terms of brackets and
central extensions on the Lie groups $SO(2)$ and $SO(3)$. }

%%\pacs[JEL Classification]{D8, H51}


\maketitle

\section{Introduction}\label{sec1}
In thermodynamics Lie groups and information geometry, the Fisher
metric plays a fundamental role in quantifying the information
contained in a statistical model. Initially introduced by Ronald
Fisher, this metric has been generalized by several authors, as Shu
ichi Amari \cite{Shu}, who proved that the Riemannian metric in an
exponential family is the Fisher information matrix defined by:
$g_{ij}=-\left[\frac{\partial^{2}\Phi}{\partial\theta_{i}\partial\theta_{i}}\right]_{ij}$
with
$\Phi(\theta)=-\log\int_{\mathrm{I\!R}}e^{-\langle\theta,y\rangle}
dy$ provide by  the dual potential $\Psi$, called Shannon entropy
given by the Legendre transform:
$\Psi(\eta)=\langle\theta,\eta\rangle-\Phi(\theta)$ with
$\eta_{i}=\frac{\partial\Phi}{\partial\theta_{i}},\;
\theta_{i}=\frac{\partial\Psi}{\partial\eta_{i}}$; Jean Louis Koszul
and Eugen Vinberg\cite{barbaresco,barbaresco1}, introduced a
generalization using a Hessian metric that is affinely invariant on
a salient convex cone $\Omega$,  by its characteristic function
$\Phi_{\Omega}(\theta)=-\log\int_{\Omega^{*}}e^{-\langle\theta,y\rangle}
dy=-\log\psi_{\Omega}(\theta)$, $\theta\in\Omega$ and
$\psi_{\Omega}(\theta)=\log\int_{\Omega^{*}}e^{-\langle\theta,y\rangle}
dy$. More recently, Jean-Marie
Souriau\cite{souriau,barbaresco3,souriau1}, studies thermodynamics
Lie groups of dynamical systems where the Gibbs density (maximum
entropy) is covariant with respect to the action of the Lie group.
In the Souriau model, the following structures are invariant with
respect to the adjoint representation:
$I(\beta)=-\frac{\partial^{2}\Phi}{\partial\beta^{2}}$ with
$\Phi(\beta)=-\log\int_{M}e^{-\langle U(\xi),\beta\rangle}
d\lambda_{w}$ where $U:M\rightarrow\mathfrak{g}^{*}$ is the moment
map and $\Psi$ the dual potential called Shannon entropy, satisfying
the equation $\Psi(\eta)=\langle\beta,\eta\rangle-\Phi(\beta)$ with
$\eta=\frac{\partial\Phi}{\partial\beta}\in\mathfrak{g}^{*},\;
\beta=\frac{\partial\Psi}{\partial\eta}\in\mathfrak{g}$.
Souriau\cite{barbaresco1}, proposed some Riemannian metric
identified as a generalization of the Fisher metric:
$I(\beta)=\left[g_{\beta}\right]$ with
\begin{equation*}
g_{\beta}\left(\left[\beta,Z_{1}\right],\left[\beta,Z_{2}\right]\right)=\tilde{\Theta}_{\beta}\left(Z_{1},\left[\beta,Z_{2}\right]\right)
\;\textrm{with}\;
\tilde{\Theta}_{\beta}\left(Z_{1},Z_{2}\right)=\tilde{\Theta}\left(Z_{1},Z_{2}\right)+\langle
\eta,ad_{Z_{1}}(Z_{2})\rangle.
\end{equation*}  He proved that any co-adjoint orbit of a Lie group given by
$O_{F}=\left\{Ad^{*}_{g}F, g\in G\right\}$ in $\mathfrak{g}^{*}$,
$F\in\mathfrak{g}^{*}$ carries a natural homogeneous symplectic
structure by a closed $G$-invariant 2-form. He shows that any
symplectic manifold on which a Lie group acts transitively by a
Hamiltonian action is a covering space of a coadjoint orbit. In this
work, we study  the Fisher metric in the framework of Souriau's
thermodynamics Lie group, on $SO(2)$ and $SO(3)$ rotation groups.
Statistical systems are often governed by symmetries, generally
described by the action of a Lie group. How are these symmetries
manifest themselves in their information structure and dynamics, and
how can the central $2$-cocycles affect the information geometry and
integrability of gradient systems? This work seeks to answer this
question in a concrete and computable way. Then by the Souriau's
formalism we compute on the Lie groups $SO(2)$ and $SO(3)$, the
2-cocycles which provide the information between Souriau-Fisher
metric and the Fisher one. In addition, we show how the algebraic
structure of a Lie group, control the dynamics (gradient systems and
integrable Hamiltonian) of the geometry of $SO(2)$ and $SO(3)$,
through cocycles and coadjoint orbits. The work of F. Barbaresco
\cite{barbaresco,barbaresco1,barbaresco2} has shown how the Souriau
Fisher metric can be applied to radar signal processing on
covariance manifolds. Our work, by providing explicit foundations on
groups such as $SO(3)$, can enable new applications in filtering and
estimation on these groups. Work in this area dates back to Amari
\cite{Shun}. Following this work, we were able to show that on the
Lie groups $SO(2)$ and $SO(3)$ there exists a unique one-cocycle of
the Lie algebra $\Theta_{\beta}:\mathfrak{so}(2)\longrightarrow
\mathfrak{so}(2):X\mapsto\frac{1}{a^{2}}X$  which is linear for
$\beta$ fixed, such that the distinguished density function is given
by
\begin{equation*}p(\beta,\xi)=\frac{e^{-\langle\Theta_{\beta}^{-1}\left(\eta\right),\xi\rangle}}{\int_{\Omega^{*}}
e^{-\langle\Theta_{\beta}^{-1}\left(\eta\right),\xi\rangle} dx},
i.e.,\; p(\beta,\xi)=\frac{e^{-2ax}}{\int^{+\infty}_{0} e^{-2ax}
dx};\end{equation*} the potential function $\Phi$ and the dual
potential function $\Psi$ satisfying the Legendre equation
$\Psi\left(\eta\right)=\langle\beta,\eta\rangle-\Phi\left(\beta\right)$
is given par
$\Phi\left(\beta\right)=\log(2a),\Psi\left(\eta\right)=1-\log(2a).$
such that $\frac{\partial\Psi\left(\eta\right)}{\partial
\eta}=\beta,\; and\; \frac{\partial\Phi\left(\beta\right)}{\partial
\beta}=\eta.$ With $\eta=\left(
                                                                                            \begin{array}{cc}
                                                                                              0 & -\frac{1}{a} \\
                                                                                              \frac{1}{a} & 0 \\
                                                                                            \end{array}
                                                                                          \right)$.
                                                                                          We were able to show that for the case of Lie groups $SO(2)$ and $SO(3)$ the $1$-cocycle $\theta(g)$ represents the equivariant action
with $\theta(g)=0$. We showed that for any $g\in SO(2)$, the
generalization of the Fisher metric of Souriau on supplementary
space $\mathfrak{p}$ is given by
                                 \begin{equation*}
I\left(Ad_{g}(\Theta_{\beta}^{-1}\left(\eta\right))\right)=I\left(\beta\right)=\left(
                                   \begin{array}{cc}
                                     4a^{2} & 0\\
                                     0 & 4a^{2}\\
                                   \end{array}
                                 \right),\end{equation*}
and there is an additional vector subspace
$\mathfrak{p}=\left\{\left(
                                                  \begin{array}{cc}
                                                    x & y \\
                                                    y & -x \\
                                                  \end{array}
                                                \right),\; x,y\in \mathrm{I\! R}
                 \right\}\subset\mathfrak{s}l_{2}(\mathrm{I\! R})$ containing the solutions $Z_{1}=\left(
                   \begin{array}{cc}
                     0 & \frac{1}{2a^{2}}\\
                    \frac{1}{2a^{2}} & 0 \\
                   \end{array}
                 \right),\; Z_{2}=\left(
                   \begin{array}{cc}
                    \frac{1}{2a^{2}} & 0 \\
                     0 & -\frac{1}{2a^{2}}\\
                   \end{array}
                 \right)$ of equation system
                  \begin{equation*}
                  \tilde{\Theta}\left(Z_{i},\left[\Theta_{\beta}^{-1}\left(\Theta_{\beta}^{-1}\left(\beta\right)\right),Z_{j}\right]\right)+
                  \langle\Theta_{\beta}\left(\beta\right),
                  \left[Z_{i},\left[\Theta_{\beta}^{-1}\left(\beta\right),Z_{j}\right]\right]\rangle=\frac{1}{a^{2}}Id_{2},
                  \;1\leq i,j \leq 2
\end{equation*} such that $\mathfrak{s}l_{2}(\mathrm{I\! R})=\mathfrak{so}(2)\oplus\mathfrak{p}$, with $\tilde{\Theta}$ is a
2-cocycle. We show that there exists a foliation whose leaves are
the orbit of the $1$-dimensional foliated action
$\mathcal{O}_{\xi}(\nu)=\left\{\left(
                                     \begin{array}{cc}
                                       0 & -x \\
                                       x & 0 \\
                                     \end{array}
                                   \right)
,\; x>0\right\}$ diffeomorphic to the $1$-dimensional torus in the
case of Lie group $SO(2)$. We show that by non-trivial Cartan
decomposition of the Lie algebra $\mathfrak{so}(3)$ the generalized
Fisher information metric on Lie group $SO(3)$ on the basis
$\left\{Z_{1},Z_{2},Z_{3}\right\}$ is given by
\begin{equation*}
    I(\beta)=\left(
             \begin{array}{ccc}
               2a^{2} & 0 & 0 \\
               0 &  2a^{2} & 0 \\
               0 & 0 & 0 \\
             \end{array}
           \right)
\end{equation*}
with $Z_{1}=\left(
             \begin{array}{ccc}
               0 & 0 & 0 \\
               0 &  0 & -1 \\
               0 & 1 & 0 \\
             \end{array}
           \right),
Z_{2}=\left(
             \begin{array}{ccc}
               0 & 0 & 1 \\
               0 &  0 & 0 \\
               -1 & 0 & 0 \\
             \end{array}
           \right),Z_{3}=\left(
             \begin{array}{ccc}
               0 & -1 & 0 \\
               1 &  0 & 0 \\
               0 & 0 & 0 \\
             \end{array}
           \right).$
This allowed us to establish the link between gradient systems and
completely integrable  systems on the Lie algebras $
\mathfrak{so}(2)$ and $ \mathfrak{so}(3)$. We show that there exists
a foliation whose leaves are the orbit of the $1$-dimensional
foliated action $\mathcal{O}_{\xi}(\nu^{*})=\left\{\left(
             \begin{array}{ccc}
               0 & -x & 0 \\
               x & 0 & 0 \\
               0 & 0 & 0 \\
             \end{array}
           \right),\;
x>0\right\}$ diffeomorphic to the $1$-dimensional torus in the case
of Lie group $SO(3)$. Our results illustrate the richness of
geometric and algebraic structures underlying invariant statistical
models, and open perspectives for the study of neural inspired
systems and neural networks in a unified geometric framework. After
the introduction, the first section recall the preliminaries, in
section $3$ we determine the distinguished density function, in
section $4$, we construct the gradient system  and we prove that it
is  completely integrable on $SO(2)$, in section $5$, we generalized
Fisher metric on Lie group $SO(3)$ and completely integrable system.

\section{Preliminaries}\label{sec2}

\subsection{Lie Algebra}
We can define Lie algebras over any field. We restrict ourselves to
the real or complex case. So let $\mathrm{I\! K}=\mathrm{I\! R}$ or
$\mathbb{C}$.

\begin{definition}\cite{Lie}A Lie algebra  $\mathfrak{a}$ over $\mathrm{I\! K}$ is a finite- or
infinite-dimensional $\mathrm{I\! K}$-vector space with a
$\mathrm{I\! K}$-bilinear, antisymmetric operation $[ , ]$
satisfying the Jacobi identity
\begin{equation}\label{c}
    X, Y, Z\in\mathfrak{a},\; [X,[Y,Z]] + [Y,[Z, X]] + [Z,[X, Y]] = 0
\end{equation}
 The bilinear operation $[ , ]$ is called the Lie bracket (or
simply bracket). The center of a Lie algebra $\mathfrak{a}$ is the
abelian ideal of $\mathfrak{a}$,\begin{equation}\label{c}
    \left\{X\in\mathfrak{a}|\forall Y\in\mathfrak{a},\;
    [X,Y]=0\right\}
\end{equation} A Lie algebra is called simple if
it has no nontrivial ideals and is not of dimension $0$ or $1$. It
is called semi simple if it has no nonzero abelian ideals.
\end{definition}

\begin{definition}\cite{Lie}A Lie subalgebra of a Lie algebra is a vector subspace closed under
the bracket\end{definition}

\begin{proposition}\cite{Lie}
The Lie algebras of $SL(n,\mathrm{I\! K}),\;O(n),\;SO(n)$ are
\begin{eqnarray*}
% \nonumber to remove numbering (before each equation)
  \mathfrak{sl}(n,K) &=& \left\{X\in
  \mathfrak{gl}(n,K)|TrX=0\right\}\end{eqnarray*}
  is the Lie algebra of traceless $n\times n$ matrices with coefficients
  $n \times n\in\mathrm{I\! K}$, and $dim_{\mathrm{I\! R}}\mathfrak{sl}(n,\mathrm{I\! R})
  =n^{2}-1,\; dim_{\mathbb{C}}\mathfrak{sl}(n,\mathbb{C})=n^{2}-1$
  , $dim_{\mathrm{I\! R}}\mathfrak{sl}(n,\mathbb{C})=2(n^{2}-1)$.
  \begin{eqnarray*}
% \nonumber to remove numbering (before each equation)
  \mathfrak{o}(n)&=&\mathfrak{so}(n)=\left\{X\in \mathfrak{gl}(n,\mathrm{I\! R}) | X+^{t}X=0\right\}
\end{eqnarray*}
 is the Lie algebra of antisymmetric real $n\times n$ matrices
and $dim \mathfrak{so}(n)=n(n-1)/2$.
\end{proposition}

\subsection{Coadjoint operator and Coadjoint Orbits (Kirillov
Representation)}
\begin{definition}\cite{souriau,souriau1}The adjoint representation of a Lie group $Ad_{g}$ is a way of representing its elements
as linear transformations of the Lie algebra, considered as a vector
space
\begin{eqnarray*}
% \nonumber to remove numbering (before each equation)
  \varphi: G&\longrightarrow & Aut(G)\\
  g &\longmapsto& \varphi_{g}(h)=ghg^{-1}
\end{eqnarray*}
\begin{eqnarray*}
% \nonumber to remove numbering (before each equation)
  Ad_{g}=\varphi_{g}(h): \mathfrak{g}&\longrightarrow & \mathfrak{g}\\
  X&\longmapsto& Ad_{g}(X)=gXg^{-1}
\end{eqnarray*}
\begin{eqnarray*}
% \nonumber to remove numbering (before each equation)
  ad_{g}=T_{e}Ad: T_{e}(G)&\longrightarrow & End\left(T_{e}(G)\right)\\
  X,Y&\longmapsto& ad_{X}(Y)=\left[X,Y\right]
\end{eqnarray*}
the coadjoint representation of a Lie group $Ad^{*}_{g}$, is the
dual of the adjoint representation ( denotes the dual space to
$\mathfrak{g}$ ):
\begin{equation*}
    \forall g\in G, Y\in \mathfrak{g}, F\in \mathfrak{g}^{*},
    \textrm{then},\; \langle Ad^{*}_{g}F,Y\rangle=\langle F,Ad_{g^{-1}}Y\rangle
\end{equation*}
A coadjoint orbit:
\begin{eqnarray*}
    O_{F}&=&\left\{Ad^{*}_{g}F=g^{-1}F g,\;g\in G\;F\in \mathfrak{g}^{*}\right\}
,\;  K=Ad^{*}_{g}=\left(Ad_{g^{-1}}\right)^{*}, and \;
K_{*}X=-\left(Ad_{X}\right)^{*}
\end{eqnarray*}
carry a natural homogeneous symplectic structure by a closed
$G$-invariant 2-form: \begin{equation*}
    \sigma_{\Omega}\left(K_{*X}F,K_{*Y}F\right)=B_{F}\left(X,Y\right)=\langle F,\left[X,Y\right]\rangle,\; X,Y\in \mathfrak{g}
\end{equation*} The coadjoint action on is a Hamiltonian
$G$-action with moment map given by
$\Omega\longrightarrow\mathfrak{g}^{*}$.
\end{definition}

\begin{definition}\cite{souriau,souriau1}
The tensor $\tilde{\Theta}$ is  defined in tangent space of the
cocycle $\theta(g)\in\mathfrak{g}^{*}$(this cocycle appears due to
the non-equivariance of the coadjoint operator $Ad^{*}_{g}$, action
of the group on the dual lie algebra):
 $\eta\left(Ad_{g}\left(\beta\right)\right)=Ad^{*}_{g}\left(\eta(\beta)\right)+\theta(g)$
\begin{eqnarray*}
% \nonumber to remove numbering (before each equation)
  \tilde{\Theta}:\mathfrak{g}\times \mathfrak{g}&\longrightarrow & \mathrm{I\!
R}\\
  X,Y&\longmapsto& \langle\Theta(X),Y\rangle
\end{eqnarray*}with $\Theta(X)=T_{e}\theta(X(e))$. According to J.M.
Souriau\cite{souriau,souriau1}, the generalized information metric
is given by $I(\beta)=\left[g_{\beta}\right]$ where
\begin{equation*}
g_{\beta}\left(\left[\beta,Z_{1}\right],\left[\beta,Z_{2}\right]\right)=\tilde{\Theta}_{\beta}\left(Z_{1},\left[\beta,Z_{2}\right]\right)
\end{equation*}
with
$\tilde{\Theta}_{\beta}\left(Z_{1},Z_{2}\right)=\tilde{\Theta}\left(Z_{1},Z_{2}\right)+\langle
\eta,ad_{Z_{1}}(Z_{2})\rangle$.
\end{definition}

\subsection{Folliations}
\begin{definition}\cite{feuille}
A foliation of dimension  $n$  of a  differentiable  manifold
$M^{m}$ is, roughly speaking,  a  decomposition  of $M$ into
connected sub-manifolds  of dimension  $n$  called leaves, which
locally stack up like  the  subsets  of  $\mathrm{I\!
R}^{m}=\mathrm{I\! R}^{n}\times \mathrm{I\! R}^{m-n}$ with the
second coordinate. The diffeomorphisms
\begin{equation*}
h:  U  \subset  \mathrm{I\! R}^{m} \longrightarrow V \subset
\mathrm{I\! R}^{m} \end{equation*}
  which preserve the leaves of this
foliation locally have the following form constant.
\begin{equation}\label{00}
h(x,y)=(h_{1}(x,y),h_{2}(y)),\; (x,y)\in \mathrm{I\! R}^{n}\times
\mathrm{I\! R}^{m-n}
\end{equation}\end{definition}

\begin{definition}\cite{feuille}
Let $M$ be  a  $C^{\infty}$  manifold of dimension  $m$.  A $C^{r}$
foliation of dimension  $n$  of $M$ is  a  $C^{r}$ atlas  a  on  $M$
which is maximal with  the  following properties:
\begin{enumerate}
    \item If  $(  U, \varphi )\in\mathcal{A}$   then  $ \varphi(U)  = U_{1} \times  U_{2}\subset  \mathrm{I\!
R}^{n}\times \mathrm{I\! R}^{m-n}$ where  $U_{1}$ and  $U_{2}$ are
open disks in $\mathrm{I\! R}^{n}$ and  $\mathrm{I\! R}^{m-n}$
respectively.
    \item If  $(  U, \varphi )$ and $(  V, \psi )\in\mathcal{A}$   are  such that  $U\cap V\neq \emptyset$ then  the change of coordinates map
      $\psi\circ \varphi^{-1}
: \varphi(U\cap V)\longrightarrow \psi(U\cap V)$ is  of the  form
(\ref{00}), that is,$\psi\circ \varphi^{-1}(x,y)= (  h_{1}
(x,y),h_{2}(y))$. We say that $M$ is foliated by a  , or that a is a
foliated structure of dimension $n$ and  class  $C^{r}$ on $M$
\end{enumerate}

\end{definition}
\begin{definition}\cite{feuille}
A  $C^{\infty}$  action of a Lie group  $G$  on a manifold $M$ is a
map $ \psi_{1}:  G \times M \rightarrow M$ such that $\psi_{1}
\left(e,x\right) = x$ and $\psi_{1}\left( g_{1} g_{ 2} ,x\right) =
\psi_{1}\left(g_{1} ,\psi_{1} \left( g_{2}  ,x\right)\right)$ for
any $g_{1},g_{2}\in G$ and $x \in M$. The orbit of a point $x \in M$
for the action  $ \psi_{1}$ is the subset
$\mathcal{O}_{x}(\psi_{1})=\left\{\psi_{1} \left(g,x\right)\in
M|g\in G\right\}$. The isotropy group of $x \in M$ is the subgroup
$\mathcal{G}_{x}(\psi_{1})=\left\{g\in G|\psi_{1}
\left(g,x\right)=x\right\}$.
\end{definition}

\begin{proposition}\cite{feuille}The  orbits of a foliated action define  the  leaves of a
foliation.\end{proposition}
\subsection{Gradient system and completely integrable Hamiltonian system.}

Let $g$ be a Riemannian structure on $S$. Any $X,Y\in
\mathfrak{X}(S)$ induce a  $C^{\infty}(S)$-linear form $g_{X}$ on
$\mathfrak{X}(S)$ defined by $g_{X}(Y)=g(X,Y)$. That is for all
$X\in\mathfrak{X}(S)$, $g_{X}\in \Omega(S) $. Where
$\mathfrak{X}(S)$ is the module of vector fields overs  $S$.
Therefore, each Riemannian structure $g$ on $S$ induces an
isomorphism
\begin{eqnarray*}
b_{g}:\mathfrak{X}(S)&\tilde{\longrightarrow}&
\Omega(S)\\
X&\longmapsto& g_{X}.
\end{eqnarray*}
  Let
 $\Phi\in C^{\infty}(S)$. We have
 $d\Phi\in \Omega(S)$ and there exists $X_{\Phi}\in \mathfrak{X}(S)$ such that $b_{g}(X_{\Phi})=d\Phi$.
As a linear mapping, $b_{g}$ has a matrix $[b_g]$  with respect to
the basis pair $\left(\partial_{\theta_{i}},\; d\theta_{i}\right)$
and we known that $(g_{ij})_{1\leq i,j\leq n}=[b_g]$. Since
$[b_g](X_{\Phi})=d\Phi$, we have $X_{\Phi}=[b_g]^{-1}d\Phi$. We
denote $X_{\Phi}=grad_{g}(\Phi)$ and  $\Phi$ is called gradient
potential function associated to the gradient field $X_{\Phi}$ with
respect to $g$. By definition, the gradient system on the Riemannian
manifold $(S,g)$ is the negative flow of the vector field
$grad_{g}(\Phi)$. Therefore, it is defined as follows
 $
\dot{\overrightarrow{\theta}}=-[b_g]^{-1}\partial_{\theta}\Phi(\theta)$;
with $\Phi$ the potential function,
$\partial_{\theta}\Phi(\theta)=\left(\partial_{\theta_{1}}\Phi(\theta),\dots,\partial_{\theta_{n}}\Phi(\theta)\right)^{T}$,
and
\begin{equation}
[b_g]=(g_{ij})_{1\leq i;j\leq n}\label{e1}
\end{equation}
\begin{definition}\cite{nakamura}
Let $S$ be a manifold. The gradient system on $S$ is given by:
\begin{eqnarray}\label{sg}\dot{\overrightarrow{\theta}}&=&-G^{-1}\partial_{\theta}\Phi(\theta),\end{eqnarray}
where $\Phi$ is potential function,
$\partial_{\theta}\Phi(\theta)=\left(\partial_{\theta_{1}}\Phi(\theta),\dots,\partial_{\theta_{n}}\Phi(\theta)\right)^{T}$,
and $G= (g_{ij})_{1\leq i;j\leq n}$ is the Fisher information.
\end{definition}

\begin{definition}\cite{lesfari2}
In physics mathematics, a gradient system  (\ref{sg}) is said to be
Hamiltonian if it exists $\barwedge$ and $\mathcal{H}$
 such that
\begin{eqnarray}\label{sh}\dot{x}(t)&=&\barwedge\frac{\partial \mathcal{H}}{\partial x},\end{eqnarray}
where $\mathcal{H}$ is a hamiltonian function and $\barwedge$ is a
bivector fields such that $[\barwedge,\barwedge]=0$ where $[\;,\;]$
is the Schouten bracket. \end{definition}

\section{Distinguished density function}\label{sec3}

\begin{proposition}\label{th1}
 Let $\beta=\left(
                                   \begin{array}{cc}
                                     0 & a \\
                                     -a & 0 \\
                                   \end{array}
                                 \right)\in \mathfrak{so}(2), \; a\in \mathrm{I\!
R}$ an element of the Lie algebra
                                 $\mathfrak{so}(2)$. Let $\Omega=\left\{\left(
                                                                                            \begin{array}{cc}
                                                                                              0 & -a \\
                                                                                              a & 0 \\
                                                                                            \end{array}
                                                                                          \right)
                                , a>0 \right\}$, and  $\Omega^{*}=\left\{\xi=\left(
                                                                                            \begin{array}{cc}
                                                                                              0 & -x \\
                                                                                              x & 0 \\
                                                                                            \end{array}
                                                                                          \right)
                                , x >0 \right\}$
the Koszul dual cone. There exists a unique one-cocycle of the Lie
algebra $\Theta_{\beta}:\mathfrak{so}(2)\longrightarrow
\mathfrak{so}(2):X\mapsto\frac{1}{a^{2}}X$  which is linear for
$\beta$ fixed, symmetric and positive such that the distinguished
density function is given by
\begin{equation}p(\beta,\xi)=\frac{e^{-\langle\Theta_{\beta}^{-1}\left(\eta\right),\xi\rangle}}{\int_{\Omega^{*}}
e^{-\langle\Theta_{\beta}^{-1}\left(\eta\right),\xi\rangle} dx},
i.e.,\; p(\beta,\xi)=\frac{e^{-2ax}}{\int^{+\infty}_{0} e^{-2ax}
dx};\end{equation} the potential function $\Phi$ and the dual
potential function $\Psi$ satisfying the Legendre equation
\begin{equation}\Psi\left(\eta\right)=\langle\beta,\eta\rangle-\Phi\left(\beta\right)\end{equation}
is given by
\begin{equation}\Phi\left(\beta\right)=\log(2a),\Psi\left(\eta\right)=1-\log(2a).\end{equation}
such that
\begin{equation}\frac{\partial\Psi\left(\eta\right)}{\partial \eta}=\beta,\; and\; \frac{\partial\Phi\left(\beta\right)}{\partial
\beta}=\eta.\end{equation} With $\eta=\left(
                                                                                            \begin{array}{cc}
                                                                                              0 & -\frac{1}{a} \\
                                                                                              \frac{1}{a} & 0 \\
                                                                                            \end{array}
                                                                                          \right)$.\end{proposition}
\begin{proof}

 Let $\beta=\left(
                                   \begin{array}{cc}
                                     0 & a \\
                                     -a & 0 \\
                                   \end{array}
                                 \right)\in \mathfrak{so}(2)$ an element of the Lie algebra
                                 $\mathfrak{so}(2)$. Let $\Omega=\left\{\left(
                                                                                            \begin{array}{cc}
                                                                                              0 & -a \\
                                                                                              a & 0 \\
                                                                                            \end{array}
                                                                                          \right)
                                , a>0 \right\}$, and

$\Omega^{*}=\left\{\xi\in \mathfrak{so}(2),\; \forall\beta \in
\mathfrak{so}(2),\;\langle\beta,\xi\rangle >0 \right\}$ i.e.,
$\Omega^{*}=\left\{\xi=\left(
                                                                                            \begin{array}{cc}
                                                                                              0 & -x \\
                                                                                              x & 0 \\
                                                                                            \end{array}
                                                                                          \right)
                                , x >0 \right\}=\Omega$ the dual cone of Koszul. Now the characteristic function of Koszul is given
                                by \begin{eqnarray*}
                                % \nonumber to remove numbering (before each equation)
                                  \chi(\beta) &=&
                                  \int^{+\infty}_{0}
                                  e^{-2ax}dx=\frac{1}{2a}
                                \end{eqnarray*}
                                Since the potential function is \begin{eqnarray*}
                                % \nonumber to remove numbering (before each equation)
                                  \Phi(\beta)&=&-\log\chi(\beta)
                                \end{eqnarray*}We obtain \begin{eqnarray*}
                                % \nonumber to remove numbering (before each equation)
                                  \Phi(\beta)&=&\log(2a)
                                \end{eqnarray*} We have\begin{eqnarray*}
                                % \nonumber to remove numbering (before each equation)
                                   \Phi(\beta)&=& \log(2a) \\
                                   &=& \log(2)+ \log(a)
                                \end{eqnarray*}
So \begin{eqnarray}\label{01}a^{2}&=&-\frac{1}{2}Tr(\beta^{2})=
\frac{1}{2}Tr(\beta^{T}\beta)=\langle\beta,\beta\rangle.\end{eqnarray}
So using (\ref{01}) we obtain the derivative of $\Phi$ by
\begin{eqnarray*}\frac{\partial\Phi\left(\beta\right)}{\partial\beta}&=&
\frac{1}{2}\frac{\beta+\beta}{\langle\beta,\beta\rangle}\\
&=&\frac{\beta}{\langle\beta,\beta\rangle}\\
&=&\frac{1}{a^{2}}\beta\\
&=&\left(
     \begin{array}{cc}
       0 & -\frac{1}{a} \\
       \frac{1}{a} & 0 \\
     \end{array}
   \right)
\end{eqnarray*}

Let $e=\left(
     \begin{array}{cc}
       0 & -1 \\
      1 & 0 \\
     \end{array}
   \right)$ the basis of Lie algebra. The general form of linear map is given by
\begin{eqnarray*}\Theta_{\beta}:\mathfrak{so}(2)&\longrightarrow&
\mathfrak{so}(2): X\mapsto\lambda X,\; with\; X=x e,\;x\in
\mathrm{I\! R}.
\end{eqnarray*}
We have $\beta=a e$ and $\eta=\frac{1}{a}e$. Find the value of
$\lambda$ such that $\Theta_{\beta}(\beta)=\eta$. We have
\begin{equation*}
    \Theta_{\beta}(ae)=\lambda ae=\frac{1}{a}e.
\end{equation*}
So, $\lambda=\frac{1}{a^{2}}$. We obtain the following application
\begin{eqnarray*}\Theta_{\beta}:\mathfrak{so}(2)&\longrightarrow&
\mathfrak{so}(2): X\mapsto\frac{1}{a^{2}}X.
\end{eqnarray*}
$\Theta_{\beta}$ is linear map for fixed  beta, because for all $X,Y
\in \mathfrak{so}(2)$ such that $X=xe,\; Y=y e$, and for all
$\alpha_{1},\alpha_{2}\in \mathrm{I\! R}$ we have:
\begin{equation*}
    \Theta_{\beta}(\alpha_{1}X+\alpha_{2}Y)=
    \alpha_{1}\Theta_{\beta}(X)+\alpha_{2}\Theta_{\beta}(Y)
\end{equation*}
 The dual potential function is determined using Legendre's
equation
\begin{eqnarray*}\Psi&=&\langle \eta, \beta\rangle-\Phi
\end{eqnarray*}
we obtain
\begin{eqnarray*}\Psi&=&1-\log(2a)
\end{eqnarray*}
So, the distinguished density function is therefore
\begin{equation*}p(\beta,\xi)=\frac{e^{-\langle\beta,\xi\rangle}}{\int^{+\infty}_{0}
e^{-\langle\beta,\xi\rangle}
d\xi}=\frac{e^{-2ax}}{\int^{+\infty}_{0} e^{-2ax} dx}\end{equation*}
we have,
\begin{eqnarray*}\int_{\Omega^{*}} p(\beta,\xi)d\xi&=&\int^{+\infty}_{0}\frac{e^{-\langle\beta,\xi\rangle}}{\int^{+\infty}_{0}
e^{-\langle\beta,\xi\rangle} d\xi}d\xi\\
&&= \frac{1}{2a}\int^{+\infty}_{0}
e^{-\langle\beta,\xi\rangle} dx\\
&=&\frac{\frac{1}{2a}}{\frac{1}{2a}}\\
&=&1\\
 &=&\frac{\int^{+\infty}_{0}
e^{-\langle\beta,\xi\rangle}d\xi}{\int^{+\infty}_{0}
e^{-\langle\beta,\xi\rangle} d\xi}
\end{eqnarray*} Furthermore,

\begin{eqnarray*}\mathrm{I\! E}_{\beta}[\xi]&=&\int^{+\infty}_{0}\xi\frac{e^{-\langle\beta,\xi\rangle}}{\int^{+\infty}_{0}
e^{-\langle\beta,\xi\rangle} d\xi}d\xi\end{eqnarray*} By setting
$\chi(\beta)=\int^{+\infty}_{0} e^{-\langle\beta,\xi\rangle}
d\xi=e^{\frac{1}{2}a^{2}-k}$ we obtain

\begin{eqnarray*}\mathrm{I\! E}_{\beta}[\xi]&=&\frac{1}{\chi(\beta)}\int^{+\infty}_{0}\xi e^{-\langle\beta,\xi\rangle}d\xi\end{eqnarray*}

\begin{eqnarray*}\frac{\partial \chi(\beta)}{\partial \beta}&=&\frac{\partial }{\partial \beta}\int^{+\infty}_{0}
e^{-\langle\beta,\xi\rangle}d\xi\end{eqnarray*} Under regularity
conditions, we have
\begin{eqnarray*}\frac{\partial }{\partial \beta}\int^{+\infty}_{0}
e^{-\langle\beta,\xi\rangle}d\xi&=&\int^{+\infty}_{0}\frac{\partial
}{\partial \beta} e^{-\langle\beta,\xi\rangle}d\xi\end{eqnarray*}
So,
\begin{eqnarray*}\frac{\partial \chi(\beta)}{\partial \beta}&=&\int^{+\infty}_{0}\frac{\partial
}{\partial \beta}
e^{-\langle\beta,\xi\rangle}d\xi=-\int^{+\infty}_{0}\xi
e^{-\langle\beta,\xi\rangle}d\xi\end{eqnarray*} So,
\begin{eqnarray*}\mathrm{I\! E}_{\beta}[\xi]&=&-\frac{1}{\chi(\beta)}\frac{\partial \chi(\beta)}{\partial \beta}\\
&=&-\frac{\partial \log \chi(\beta)}{\partial
\beta}\\
&=&\frac{\partial \Phi(\beta)}{\partial \beta}\\
&=&\eta(\beta)\\
&=&\left(
     \begin{array}{cc}
       0 & -\frac{1}{a^{2}} \\
       \frac{1}{a^{2}} & 0 \\
     \end{array}
   \right)
\end{eqnarray*}

\begin{eqnarray*}\mathrm{I\! E}_{\beta}[\xi]&=&\int^{+\infty}_{0}\xi\frac{e^{-\langle\beta,\xi\rangle}}{\int^{+\infty}_{0}
e^{-\langle\beta,\xi\rangle} d\xi}d\xi= \frac{\int^{+\infty}_{0}\xi
e^{-\langle\beta,\xi\rangle}d\xi}{\int^{+\infty}_{0}
e^{-\langle\beta,\xi\rangle} d\xi}=\eta(\beta).\end{eqnarray*}
\end{proof}

\begin{theorem}\label{th2}
 Let $\beta=\left(
                                   \begin{array}{cc}
                                     0 & -a \\
                                     a & 0 \\
                                   \end{array}
                                 \right)\in \mathfrak{so}(2), \; a\in \mathrm{I\!
R}$ an element of the Lie algebra
                                 $\mathfrak{so}(2)$.
                                 Let
                                 $SO(2)=\left\{\left(
                                   \begin{array}{cc}
                                     \cos t & -\sin t\\
                                     \sin t & \cos t\\
                                   \end{array}
                                 \right),\left(
                                   \begin{array}{cc}
                                     \cos t & \sin t\\
                                     -\sin t & \cos t\\
                                   \end{array}
                                 \right)\right\}$ the Lie group
rotation.
\begin{equation*}
\forall\; g\in SO(2), \;
\eta\left(Ad_{g}(\beta)\right)=Ad^{*}_{g}\left(\eta(\beta)\right).\end{equation*}
Then the unique one-cocycle of groups $\theta(g)$ is given by
                                 $\theta(g)=\left(0_{ij}
                                            \right)_{1\leq i,j\leq2}$. In this case, the one-cocycle of Lie is given by
                                             \begin{equation*}
                                             \Theta(X)=0,\; \forall\;X\in \mathfrak{so}(2).\end{equation*}

\end{theorem}
\begin{proof}
Let $g=\left(
                                   \begin{array}{cc}
                                     \cos t & -\sin t\\
                                     \sin t & \cos t\\
                                   \end{array}
                                 \right).$
We have $Ad_{g}(\beta)=g \beta g^{-1}=\left(
                                   \begin{array}{cc}
                                     \cos t & -\sin t\\
                                     \sin t & \cos t\\
                                   \end{array}
                                 \right)\left(
                                   \begin{array}{cc}
                                     0 & -a\\
                                     a & 0\\
                                   \end{array}
                                 \right)\left(
                                   \begin{array}{cc}
                                     \cos t & \sin t\\
                                     -\sin t & \cos t\\
                                   \end{array}
                                 \right)=\left(
                                   \begin{array}{cc}
                                     0 & -a\\
                                     a & 0\\
                                   \end{array}
                                 \right).$
So, \begin{eqnarray*}
\eta\left(Ad_{g}(\beta)\right)&=&\eta\left(Ad_{g}^{-1}(\beta)\right)+
\theta(g)\\
\eta\left(\beta)\right)&=&\eta\left(\beta)\right)+
\theta(g).\end{eqnarray*} We obtain $\theta(g)=\left(0_{ij}
                                            \right)_{1\leq
                                            i,j\leq2}$. As for all  $X\in \mathfrak{so}(2)
                                            ,\;  \Theta(X)=T_{e}\theta(X(e))$,
                                            then  $\Theta(X)=0$.
\end{proof}

\section{Gradient system on Lie algebra $\mathfrak{so}(2)$}\label{sec4}

\begin{theorem}\label{th4}Let
$\Theta_{\beta}:\mathfrak{so}(2)\longrightarrow
\mathfrak{so}(2):X\mapsto\frac{1}{a^{2}}X$ the unique one-cocycle of
the Lie algebra which is linear for $\beta$ fixed
 Let $\beta=\left(
                                   \begin{array}{cc}
                                     0 & -a \\
                                     a & 0 \\
                                   \end{array}
                                 \right)\in \mathfrak{so}(2), \; a\in \mathrm{I\!
R}$ an element of the Lie algebra $\mathfrak{so}(2)$.
                                 Let
                                 $SO(2)=\left\{\left(
                                   \begin{array}{cc}
                                     \cos t & -\sin t\\
                                     \sin t & \cos t\\
                                   \end{array}
                                 \right),\left(
                                   \begin{array}{cc}
                                     \cos t & \sin t\\
                                     -\sin t & \cos t\\
                                   \end{array}
                                 \right)\right\}$ the Lie group
rotation.   For any $ g\in SO(2)$, the generalization of the Fisher
metric of  Souriau on supplementary space $\mathfrak{p}$ is given by
                                 \begin{equation}
I\left(Ad_{g}(\Theta_{\beta}^{-1}\left(\eta\right))\right)=I\left(\beta\right)=\left(
                                   \begin{array}{cc}
                                     4a^{2} & 0\\
                                     0 & 4a^{2}\\
                                   \end{array}
                                 \right),\end{equation}
and there is an additional vector subspace
$\mathfrak{p}=\left\{\left(
                                                  \begin{array}{cc}
                                                    x & y \\
                                                    y & -x \\
                                                  \end{array}
                                                \right),\; x,y\in \mathrm{I\! R}
                 \right\}\subset\mathfrak{s}l_{2}(\mathrm{I\! R})$ containing the solutions $Z_{1}=\left(
                   \begin{array}{cc}
                     0 & \frac{1}{2a^{2}}\\
                    \frac{1}{2a^{2}} & 0 \\
                   \end{array}
                 \right),\; Z_{2}=\left(
                   \begin{array}{cc}
                    \frac{1}{2a^{2}} & 0 \\
                     0 & -\frac{1}{2a^{2}}\\
                   \end{array}
                 \right)$ of equation system
                  \begin{equation}
                  \tilde{\Theta}\left(Z_{i},\left[\Theta_{\beta}^{-1}\left(\Theta_{\beta}^{-1}\left(\beta\right)\right),Z_{j}\right]\right)+
                  \langle\Theta_{\beta}\left(\beta\right),
                  \left[Z_{i},\left[\Theta_{\beta}^{-1}\left(\beta\right),Z_{j}\right]\right]\rangle=\frac{1}{a^{2}}Id_{2},
                  \;1\leq i,j \leq 2
\end{equation} such that $\mathfrak{s}l_{2}(\mathrm{I\! R})=\mathfrak{so}(2)\oplus\mathfrak{p}$, with $\tilde{\Theta}$ is a
$2$-cocycle.
\end{theorem}

\begin{proof}We know that  \begin{eqnarray*}
% \nonumber to remove numbering (before each equation)
  I(\beta) &=& -\left[\frac{\partial^{2}\Phi(\beta)}{\partial\beta^{2}}\right] \\
   &=& \frac{1}{a^{2}}.
\end{eqnarray*}
 Let $\mathfrak{so}(2)$ be the maximal Lie
algebra. Let us look for an additional vector subspace
$\mathfrak{p}$ such that $\mathfrak{s}l_{2}(\mathrm{I\!
R})=\mathfrak{so}(2)\oplus\mathfrak{p}$. Let $u=\left(
                                                   \begin{array}{cc}
                                                     0 & -1 \\
                                                     1 & 0 \\
                                                   \end{array}
                                                 \right),v=\left(
                                                   \begin{array}{cc}
                                                     0 & 1 \\
                                                     -1 & 0 \\
                                                   \end{array}
                                                 \right)\in
\mathfrak{so}(2) $, let's look for a basis $e_{1},e_{2}$ of
$\mathfrak{p}$ we have the following identities
\begin{equation*}
[e_{1},e_{2} ]\in \mathfrak{so}(2),\;[u ,e_{1} ]\in
\mathfrak{p},\;[u ,e_{2} ]\in \mathfrak{p},\;[u ,v ]\in
\mathfrak{so}(2)
\end{equation*}

Therefore $tr(e_{1})=tr(e_{2})=0$ because $e_{1},e_{2}\in
\mathfrak{s}l_{2}(\mathrm{I\! R})$. We have $e_{1}=\left(
                                                   \begin{array}{cc}
                                                     0 & 1 \\
                                                     1 & 0 \\
                                                   \end{array}
                                                 \right),e_{2}=\left(
                                                   \begin{array}{cc}
                                                     1 & 0 \\
                                                     0 & -1 \\
                                                   \end{array}
                                                 \right)$
So,  \begin{eqnarray*}[e_{1},e_{2} ]&=&\left(
                                                   \begin{array}{cc}
                                                     0 & 1 \\
                                                     1 & 0 \\
                                                   \end{array}
                                                 \right)\left(
                                                   \begin{array}{cc}
                                                     1 & 0 \\
                                                     0 & -1 \\
                                                   \end{array}
                                                 \right)-\left(
                                                   \begin{array}{cc}
                                                     1 & 0 \\
                                                     0 & -1 \\
                                                   \end{array}
                                                 \right)\left(
                                                   \begin{array}{cc}
                                                     0 & 1 \\
                                                     1 & 0 \\
                                                   \end{array}
                                                 \right)=2\left(
                                                   \begin{array}{cc}
                                                     0 & -1 \\
                                                     1 & 0 \\
                                                   \end{array}
                                                 \right)\end{eqnarray*}
 \begin{eqnarray*}[e_{2},e_{1}]&=&\left(
                                                   \begin{array}{cc}
                                                     1 & 0 \\
                                                     0 & -1 \\
                                                   \end{array}
                                                 \right)\left(
                                                   \begin{array}{cc}
                                                     0 & 1 \\
                                                     1 & 0 \\
                                                   \end{array}
                                                 \right)-\left(
                                                   \begin{array}{cc}
                                                     0 & 1 \\
                                                     1 & 0 \\
                                                   \end{array}
                                                 \right)\left(
                                                   \begin{array}{cc}
                                                     1 & 0 \\
                                                     0 & -1 \\
                                                   \end{array}
                                                 \right)=-2\left(
                                                   \begin{array}{cc}
                                                     0 & -1 \\
                                                     1 & 0 \\
                                                   \end{array}
                                                 \right)
                                                 \end{eqnarray*}
                                                 \begin{eqnarray*}xe_{1}+ye_{2}&=&\left(
                                                   \begin{array}{cc}
                                                     x & y \\
                                                     y & -x \\
                                                   \end{array}
                                                 \right)\end{eqnarray*}
Then we obtain $\mathfrak{p}=\left\{\left(
                                                  \begin{array}{cc}
                                                    x & y \\
                                                    y & -x \\
                                                  \end{array}
                                                \right),\; x,y\in \mathrm{I\! R}
                 \right\}$.

Therefore,
\begin{eqnarray*}
% \nonumber to remove numbering (before each equation)
  \left[\beta,e_{1}\right] &=& \left(
                                                   \begin{array}{cc}
                                                     0 & -a \\
                                                     a & 0 \\
                                                   \end{array}
                                                 \right)\left(
                                                   \begin{array}{cc}
                                                     0 & 1 \\
                                                     1 & 0 \\
                                                   \end{array}
                                                 \right)-\left(
                                                   \begin{array}{cc}
                                                     0 & 1 \\
                                                     1 & 0 \\
                                                   \end{array}
                                                 \right)\left(
                                                   \begin{array}{cc}
                                                     0 & -a \\
                                                     a & 0 \\
                                                   \end{array}
                                                 \right)=\left(
                                                   \begin{array}{cc}
                                                     -2a & 0 \\
                                                     0 & 2a \\
                                                   \end{array}
                                                 \right)=-2ae_{2} \\
  \left[\beta,e_{2}\right]  &=& \left(
                                                   \begin{array}{cc}
                                                     0 & -a \\
                                                     a & 0 \\
                                                   \end{array}
                                                 \right)\left(
                                                   \begin{array}{cc}
                                                     1 & 0 \\
                                                     0 & -1 \\
                                                   \end{array}
                                                 \right)-\left(
                                                   \begin{array}{cc}
                                                     1 & 0 \\
                                                     0 & -1 \\
                                                   \end{array}
                                                 \right)\left(
                                                   \begin{array}{cc}
                                                     0 & -a \\
                                                     a & 0 \\
                                                   \end{array}
                                                 \right)=\left(
                                                   \begin{array}{cc}
                                                     0 & 2a \\
                                                     2a & 0 \\
                                                   \end{array}
                                                 \right)=2ae_{1}\\
                                                 \left[e_{1},\left[\beta,e_{2}\right]\right]&=&\left[e_{1},2ae_{1}\right]=2a\left[e_{1},e_{1}\right]=0\\
 \left[e_{2},\left[\beta,e_{1}\right]\right]&=&\left[e_{2},-2ae_{2}\right]=-2a\left[e_{2},e_{2}\right]=0\\
 \left[e_{1},\left[\beta,e_{1}\right]\right]&=&-2a\left[e_{1},e_{2}\right]=-4au\\
 \left[e_{2},\left[\beta,e_{2}\right]\right]&=&2a\left[e_{2},e_{1}\right]=-4au
\end{eqnarray*}
So,
\begin{eqnarray*}
         g_{\beta}\left(\left[\beta,e_{1}\right],\left[\beta,e_{2}\right]\right) &=&
         \tilde{\Theta}\left(e_{1},\left[\beta,e_{2}\right]\right)+\langle\eta, \left[e_{1},\left[\beta,e_{2}\right]\right]\rangle \\
         &=&2a\tilde{\Theta}\left(e_{1},e_{1}\right)+\langle\eta, \left[e_{1},\left[\beta,e_{2}\right]\right]\rangle \\
        &=&0 \\
         g_{\beta}\left(\left[\beta,e_{2}\right],\left[\beta,e_{1}\right]\right) &=&
         \tilde{\Theta}\left(e_{2},\left[\beta,e_{1}\right]\right)+\langle\eta, \left[e_{2},\left[\beta,e_{1}\right]\right]\rangle \\
         &=&-2a\tilde{\Theta}\left(e_{2},e_{2}\right)+\langle\eta, \left[e_{2},\left[\beta,e_{1}\right]\right]\rangle \\
         &=&0 \\
         g_{\beta}\left(\left[\beta,e_{1}\right],\left[\beta,e_{1}\right]\right) &=&
         \tilde{\Theta}\left(e_{1},\left[\beta,e_{1}\right]\right)+\langle\eta, \left[e_{1},\left[\beta,e_{1}\right]\right]\rangle \\
         &=&-2a\tilde{\Theta}\left(e_{1},e_{2}\right)+\langle\eta, -4au\rangle \\
 &=&4a^{2} \\
         g_{\beta}\left(\left[\beta,e_{2}\right],\left[\beta,e_{2}\right]\right) &=&
         \tilde{\Theta}\left(e_{2},\left[\beta,e_{2}\right]\right)+\langle\eta, \left[e_{2},\left[\beta,e_{2}\right]\right]\rangle \\
         &=&2a\tilde{\Theta}\left(e_{2},e_{1}\right)+\langle\eta, -4au\rangle \\
     &=&4a^{2}
\end{eqnarray*}with $\tilde{\Theta}\left(e_{2},e_{1}\right)=0=-\tilde{\Theta}\left(e_{1},e_{2}\right)$. We obtain the generalization of the Fisher
metric of  Souriau given by
                                 \begin{equation*}
I\left(Ad_{g}(\Theta^{-1}\left(\beta\right))\right)=I(\Theta^{-1}\left(\beta\right))=\left(
                                   \begin{array}{cc}
                                     4a^{2} & 0\\
                                     0 & 4a^{2}\\
                                   \end{array}
                                 \right).\end{equation*}

Normalizing to have $1$ on the last two equations we will have
$Z_{1}=\left(
                   \begin{array}{cc}
                     0 & \frac{1}{2a^{2}} \\
                    \frac{1}{2a^{2}} & 0 \\
                   \end{array}
                 \right),\; Z_{2}=\left(
                   \begin{array}{cc}
                     \frac{1}{2a^{2}} & 0 \\
                     0 & -\frac{1}{2a^{2}}\\
                   \end{array}
                 \right)$, so

                 \begin{eqnarray*}
% \nonumber to remove numbering (before each equation)
  \left[\beta,e_{1}\right] &=& \left(
                                                   \begin{array}{cc}
                                                     0 & -a \\
                                                     a & 0 \\
                                                   \end{array}
                                                 \right)\left(
                                                   \begin{array}{cc}
                                                     0 & \frac{1}{2a^{2}} \\
                                                     \frac{1}{2a^{2}} & 0 \\
                                                   \end{array}
                                                 \right)-\left(
                                                   \begin{array}{cc}
                                                     0 & \frac{1}{2a^{2}}  \\
                                                     \frac{1}{2a^{2}}  & 0 \\
                                                   \end{array}
                                                 \right)\left(
                                                   \begin{array}{cc}
                                                     0 & -a \\
                                                     a & 0 \\
                                                   \end{array}
                                                 \right)=\left(\begin{array}{cc}
                                                     -\frac{1}{2a^{2}} & 0  \\
                                                    0 & \frac{1}{2a^{2}} \\
                                                   \end{array}
                                                 \right)=- \frac{1}{a}e_{2} \\
  \left[\beta,Z_{2}\right]  &=& \left(
                                                   \begin{array}{cc}
                                                     0 & -a \\
                                                     a & 0 \\
                                                   \end{array}
                                                 \right)\left(
                                                   \begin{array}{cc}
                                                     \frac{1}{2a^{2}}  & 0 \\
                                                     0 & -\frac{1}{2a^{2}}  \\
                                                   \end{array}
                                                 \right)-\left(
                                                   \begin{array}{cc}
                                                     \frac{1}{2a^{2}}  & 0 \\
                                                     0 & -\frac{1}{2a^{2}}  \\
                                                   \end{array}
                                                 \right)\left(
                                                   \begin{array}{cc}
                                                     0 & -a \\
                                                     a & 0 \\
                                                   \end{array}
                                                 \right)=\left(
                                                   \begin{array}{cc}
                                                     0 & \frac{1}{2a^{2}}  \\
                                                     \frac{1}{2a^{2}}  & 0 \\
                                                   \end{array}
                                                 \right)= \frac{1}{a} e_{1}\\
  \left[Z_{1},\left[\beta,Z_{2}\right]\right]&=&\left[\frac{1}{2a^{2}} e_{1},\frac{1}{a}e_{1}\right]=\frac{-1}{2a^{3}}\left[e_{1},e_{1}\right]=0\\
 \left[Z_{2},\left[\beta,Z_{1}\right]\right]&=&\left[\frac{1}{2a^{2}} e_{2},-\frac{1}{a}e_{2}\right]=-\frac{1}{2a^{3}}\left[e_{2},e_{2}\right]=0\\
 \left[Z_{1},\left[\beta,Z_{1}\right]\right]&=&\left[\frac{1}{2a^{2}} e_{1},-\frac{1}{a}e_{2}\right]=-\frac{1}{2a^{3}}\left[e_{1},e_{2}\right]=-\frac{1}{a^{3}}u\\
 \left[Z_{2},\left[\beta,Z_{2}\right]\right]&=&\left[\frac{1}{2a^{2}} e_{2},\frac{1}{a}e_{1}\right]=\frac{1}{2a^{3}}\left[e_{2},e_{1}\right]=-\frac{1}{a^{3}}u\\
\end{eqnarray*}
We obtain
                 \begin{eqnarray*}
         g_{\beta}\left(\left[\beta,Z_{1}\right],\left[\beta,Z_{2}\right]\right) &=&
         \tilde{\Theta}\left(Z_{1},\left[\beta,Z_{2}\right]\right)+\langle\eta, \left[Z_{1},\left[\beta,Z_{2}\right]\right]\rangle \\
         &=&\tilde{\Theta}\left(\frac{1}{2a^{2}}e_{1},\frac{1}{a}e_{1}\right) \\
         &=&\frac{1}{2a^{3}}\tilde{\Theta}\left(e_{1},e_{1}\right) \\
         &=&0\\
         g_{\beta}\left(\left[\beta,Z_{2}\right],\left[\beta,Z_{1}\right]\right) &=&
         \tilde{\Theta}\left(Z_{2},\left[\beta,Z_{1}\right]\right)+\langle\eta, \left[Z_{2},\left[\beta,Z_{1}\right]\right]\rangle \\
         &=&\tilde{\Theta}\left(\frac{1}{2a^{2}}e_{2},-\frac{1}{a}e_{2}\right) \\
         &=&-\frac{1}{2a^{3}}\tilde{\Theta}\left(e_{2},e_{2}\right) \\
         &=&0\\
         g_{\beta}\left(\left[\beta,Z_{1}\right],\left[\beta,Z_{1}\right]\right) &=&
         \frac{1}{2a^{3}}\tilde{\Theta}\left(e_{1},e_{2}\right)+\langle\eta, -\frac{1}{a^{3}}u\rangle \\
         &=&\frac{1}{a^{2}}\\
         g_{\beta}\left(\left[\beta,Z_{2}\right],\left[\beta,Z_{2}\right]\right) &=&
         \frac{1}{2a^{3}}\tilde{\Theta}\left(e_{2},e_{1}\right)
         +\langle\eta, -\frac{1}{a^{3}}u\rangle \\
         &=&\frac{1}{a^{2}}
\end{eqnarray*}with $\Theta^{-1}\left(\eta\right)=\beta=\Theta\left(\beta\right),\; because,\; \beta=\eta. $
who check the system below
\begin{equation*}
    \left\{
      \begin{array}{ll}
         g_{\beta}\left(\left[\Theta_{\beta}^{-1}\left(\eta\right),Z_{1}\right],\left[\Theta_{\beta}^{-1}\left(\eta\right),Z_{1}\right]\right) =\frac{1}{a^{2}}& \hbox{} \\
         g_{\beta}\left(\left[\Theta_{\beta}^{-1}\left(\eta\right),Z_{1}\right],\left[\Theta_{\beta}^{-1}\left(\eta\right),Z_{2}\right]\right) =0& \hbox{} \\
         g_{\beta}\left(\left[\Theta_{\beta}^{-1}\left(\eta\right),Z_{2}\right],\left[\Theta_{\beta}^{-1}\left(\eta\right),Z_{1}\right]\right) =0& \hbox{} \\
        g_{\beta}\left(\left[\Theta_{\beta}^{-1}\left(\eta\right),Z_{2}\right],\left[\Theta_{\beta}^{-1}\left(\eta\right),Z_{2}\right]\right) =\frac{1}{a^{2}} & \hbox{.}
      \end{array}
    \right.
\end{equation*}

we obtain $I(\beta)=\left(
           \begin{array}{cc}
             g_{\beta}\left(\left[\Theta_{\beta}^{-1}\left(\eta\right),Z_{1}\right],\left[\Theta_{\beta}^{-1}\left(\eta\right),Z_{1}\right]\right) &
             g_{\beta}\left(\left[\Theta_{\beta}^{-1}\left(\eta\right),Z_{1}\right],\left[\Theta_{\beta}^{-1}\left(\eta\right),Z_{2}\right]\right) \\
             g_{\beta}\left(\left[\Theta_{\beta}^{-1}\left(\eta\right),Z_{2}\right],\left[\Theta_{\beta}^{-1}\left(\eta\right),Z_{1}\right]\right) &
             g_{\beta}\left(\left[\Theta_{\beta}^{-1}\left(\eta\right),Z_{2}\right],\left[\Theta_{\beta}^{-1}\left(\eta\right),Z_{2}\right]\right) \\
           \end{array}
         \right)=\left(
                   \begin{array}{cc}
                     \frac{1}{a^{2}} & 0 \\
                     0 & \frac{1}{a^{2}} \\
                   \end{array}
                 \right)
         $
\end{proof}

\begin{theorem}\label{th5}
Let the unique one-cocycle of the Lie algebra
$\Theta_{\beta}:\mathfrak{so}(2)\longrightarrow
\mathfrak{so}(2):X\mapsto\frac{1}{a^{2}}X$  which is linear for
$\beta$ fixed. Consider $\Omega^{*}=\left\{\xi=\left(
                                                                                            \begin{array}{cc}
                                                                                              0 & -x \\
                                                                                              x & 0 \\
                                                                                            \end{array}
                                                                                          \right)
                                , x >0 \right\}$
the Koszul dual cone, and $\nu:SO(2)\times
\Omega^{*}\longrightarrow\Omega^{*}:(g,\xi)\longmapsto
Ad^{*}_{g}(\xi)$ the action of Lie group $SO(2)$. Let $\beta=\left(
                                   \begin{array}{cc}
                                     0 & -a \\
                                     a & 0 \\
                                   \end{array}
                                 \right)\in \mathfrak{so}(2), \; a\in \mathrm{I\!
R}$ an element of the Lie algebra  $\mathfrak{so}(2)$.
   The gradient system is given by
\begin{equation}\label{ha}
\dot{\beta}=-\beta\end{equation} where its reduced form is given by
\begin{equation}\label{haa1} \dot{a}=-a.\end{equation}
The gradient system (\ref{haa1}) is not Hamiltonian on the phase
space considered and integrable by quadrature of explicit solution
$a(t)=\lambda e^{-t}\; \lambda\in \mathrm{I\! R}$. There exists a
foliation whose leaves are the orbit of the $1$-dimensional foliated
action $\mathcal{O}_{\xi}(\nu)=\left\{\left(
              \begin{array}{cc}
                0 & -x \\
                x & 0 \\
              \end{array}
            \right),\; x>0\right\}$ diffeomorphic to the
$1$-dimensional torus. In this case the system (\ref{haa1}) is
completely integrable. The generalized gradient system on $SO(2)$ is
given by
\begin{equation*} \dot{\Theta}_{\beta}^{-1}(\eta)=\mathcal{K}(a)\eta.\end{equation*}
With $\mathcal{K}(a)=\left(
                                   \begin{array}{cc}
                                     -\frac{1}{4}a^{2} & 0\\
                                     0 & -\frac{1}{4}a^{2}\\
                                   \end{array}
                                 \right)$, where
         $\eta=\Theta_{\beta}\left(\beta\right).$
\end{theorem}

\begin{proof}The defined gradient system will be given by
\begin{equation}
\dot{\beta}=-a^{2}\partial_{\beta}\Phi(\beta)\rangle\end{equation}
where $\partial_{\beta}:=\frac{\partial}{\partial \beta}$. So,
\begin{equation}\label{a}
\dot{\beta}=-\beta=\left(
                                           \begin{array}{cc}
                                             0 & a \\
                                             -a & 0 \\
                                           \end{array}
                                         \right)\end{equation}

In the same,
\begin{equation*}
\dot{\beta}=\frac{\partial \beta}{\partial a}\dot{a}\end{equation*}
We obtain  \begin{equation} \dot{\beta}=\left(
                                           \begin{array}{cc}
                                             0 & -\dot{a} \\
                                             \dot{a} & 0 \\
                                           \end{array}
                                         \right)\end{equation}
Thus by identification with \ref{a}, we have \begin{equation*}
\dot{a}=-a.\end{equation*} Let  $\nu:SO(2)\times
\Omega^{*}\longrightarrow\Omega^{*}:(g,\xi)\longmapsto
Ad^{*}_{g}(\xi)$, us first show that $\vartheta$ is action.
 Let
                                 $SO(2)=\left\{\left(
                                   \begin{array}{cc}
                                     \cos t & -\sin t\\
                                     \sin t & \cos t\\
                                   \end{array}
                                 \right),\left(
                                   \begin{array}{cc}
                                     \cos t & \sin t\\
                                     -\sin t & \cos t\\
                                   \end{array}
                                 \right)\right\}$ the Lie group
rotation.   For any $ g\in SO(2)$, show that
\begin{equation*}\nu(g,\xi)=\xi.\end{equation*} We have
\begin{equation*}\nu(g,\xi)=Ad^{*}_{g}(\xi)=\left(
                                   \begin{array}{cc}
                                     \cos t & -\sin t\\
                                     \sin t & \cos t\\
                                   \end{array}
                                 \right)\left(
              \begin{array}{cc}
                0 & -x \\
                x & 0 \\
              \end{array}
            \right)\left(
                                   \begin{array}{cc}
                                     \cos t & \sin t\\
                                     -\sin t & \cos t\\
                                   \end{array}
                                 \right)=\left(
              \begin{array}{cc}
                0 & -x \\
                x & 0 \\
              \end{array}
            \right)\end{equation*}
     Let $g_{1}=\left(
                                   \begin{array}{cc}
                                     \cos t & -\sin t\\
                                     \sin t & \cos t\\
                                   \end{array}
                                 \right),\;g_{2}=\left(
                                   \begin{array}{cc}
                                     \cos t & -\sin t\\
                                     \sin t & \cos t\\
                                   \end{array}
                                 \right)$,\; show that
                                 $\nu\left(g_{1}g_{2},\xi\right)=\nu\left(g_{1},\nu\left(g_{2},\xi\right)\right)$.
We have

\begin{equation*}
    \nu\left(g_{1}g_{2},\xi\right)=\left(g_{1}g_{2}\right)^{-1}\xi\left(g_{1}g_{2}\right).\end{equation*}
So,
\begin{equation}\label{c}\nu\left(g_{1}g_{2},\xi\right)=g_{2}^{-1}\left(g_{1}^{-1}\xi
g_{1}\right)g_{2}=g_{2}^{-1}\xi g_{2}=\xi.
\end{equation}
 In same, we have
\begin{equation}\label{c1}\nu\left(g_{1},\nu\left(g_{2},\xi\right)\right)=\nu\left(g_{1},\xi\right)=\xi.\end{equation}
Using (\ref{c}) and (\ref{c1}) we obtain
$\nu\left(g_{1}g_{2},\xi\right)=\nu\left(g_{1},\nu\left(g_{2},\xi\right)\right)$.
\end{proof}

\section{Generalized Fisher metric on Lie group $SO(3)$}\label{sec4}

\begin{proposition}\label{th8}
 Let $\beta=\left(
             \begin{array}{ccc}
               0 & -a & 0 \\
               a & 0 & 0 \\
               0 & 0 & 0 \\
             \end{array}
           \right)\in \mathfrak{so}(3), \; a\in \mathrm{I\!
R}$ an element of the Lie algebra
                                 $\mathfrak{so}(3)$. Let $\Omega=\left\{\left(
             \begin{array}{ccc}
               0 & -a & 0 \\
               a & 0 & 0 \\
               0 & 0 & 0 \\
             \end{array}
           \right)
                                , a>0 \right\}$, and  $\Omega^{*}=\left\{\xi=\left(
             \begin{array}{ccc}
               0 & -x & 0 \\
               x & 0 & 0 \\
               0 & 0 & 0 \\
             \end{array}
           \right)
                                , x >0 \right\}$
the Koszul dual cone. There exist a unique one-cocycle of the Lie
algebra $\Theta_{\beta}:\mathfrak{so}(3)\longrightarrow
\mathfrak{so}(3):X\mapsto\frac{1}{a^{2}}X$  which is linear for
$\beta$ fixed such that the distinguished density function is given
by
\begin{equation}p(\beta,\xi)=\frac{e^{-\langle\Theta_{\beta}^{-1}\left(\eta\right),\xi\rangle}}{\int_{\Omega^{*}}
e^{-\langle\Theta_{\beta}^{-1}\left(\eta\right),\xi\rangle} dx},
i.e.,\; p(\beta,\xi)=\frac{e^{-2ax}}{\int^{+\infty}_{0} e^{-2ax}
dx};\end{equation} the potential function $\Phi$ and the dual
potential function $\Psi$ satisfying the Legendre equation
\begin{equation}\Psi\left(\eta\right)=\langle\beta,\eta\rangle-\Phi\left(\beta\right)\end{equation}
is given by
\begin{equation}\Phi\left(\beta\right)=\log(2a),\Psi\left(\eta\right)=1-\log(2a).\end{equation}
such that
\begin{equation}\frac{\partial\Psi\left(\eta\right)}{\partial \eta}=\beta,\; and\; \frac{\partial\Phi\left(\beta\right)}{\partial
\beta}=\eta.\end{equation} With $\eta=\left(
             \begin{array}{ccc}
               0 & -\frac{1}{a} & 0 \\
               \frac{1}{a} & 0 & 0 \\
               0 & 0 & 0 \\
             \end{array}
           \right)$.\end{proposition}
\begin{proof}

 Let $\beta=\left(
             \begin{array}{ccc}
               0 & -a & 0 \\
               a & 0 & 0 \\
               0 & 0 & 0 \\
             \end{array}
           \right)\in \mathfrak{so}(3)$ an element of the Lie algebra
                                 $\mathfrak{so}(3)$. Let $\Omega=\left\{\left(
             \begin{array}{ccc}
               0 & -a & 0 \\
               a & 0 & 0 \\
               0 & 0 & 0 \\
             \end{array}
           \right)
                                , a>0 \right\}$, and

$\Omega^{*}=\left\{\xi\in \mathfrak{so}(3),\; \forall\beta \in
\mathfrak{so}(3),\;\langle\beta,\xi\rangle >0 \right\}$ i.e.,
$\Omega^{*}=\left\{\xi=\left(
             \begin{array}{ccc}
               0 & -x & 0 \\
               x & 0 & 0 \\
               0 & 0 & 0 \\
             \end{array}
           \right)
                                , x >0 \right\}$ the dual cone of Koszul. Now the characteristic function of Koszul is given
                                by \begin{eqnarray*}
                                % \nonumber to remove numbering (before each equation)
                                  \chi(\beta) &=&
                                  \int^{+\infty}_{0}
                                  e^{-2ax}dx=\frac{1}{2a}
                                \end{eqnarray*}
                                Since the potential function is \begin{eqnarray*}
                                % \nonumber to remove numbering (before each equation)
                                  \Phi(\beta)&=&-\log\chi(\beta)
                                \end{eqnarray*}We obtain \begin{eqnarray*}
                                % \nonumber to remove numbering (before each equation)
                                  \Phi(\beta)&=&\log(2a)
                                \end{eqnarray*} We have\begin{eqnarray*}
                                % \nonumber to remove numbering (before each equation)
                                   \Phi(\beta)&=& \log(2a) \\
                                   &=& \log(2)+ \log(a)
                                \end{eqnarray*}
So \begin{eqnarray}\label{01}a^{2}&=&-\frac{1}{2}Tr(\beta^{2})=
\frac{1}{2}Tr(\beta^{T}\beta)=\langle\beta,\beta\rangle.\end{eqnarray}
So using (\ref{01}) we obtain the derivative of $\Phi$ by
\begin{eqnarray*}\frac{\partial\Phi\left(\beta\right)}{\partial\beta}&=&
\frac{1}{2}\frac{\beta+\beta}{\langle\beta,\beta\rangle}\\
&=&\frac{\beta}{\langle\beta,\beta\rangle}\\
&=&\frac{1}{a^{2}}\beta\\
&=&\left(
             \begin{array}{ccc}
               0 & -\frac{1}{a} & 0 \\
               \frac{1}{a} & 0 & 0 \\
               0 & 0 & 0 \\
             \end{array}
           \right)
\end{eqnarray*}

Let $\left\{Z_{1},Z_{2},Z_{3}\right\}$ the basis of Lie algebra with
$$Z_{1}=\left(
             \begin{array}{ccc}
               0 & 0 & 0 \\
               0 &  0 & -1 \\
               0 & 1 & 0 \\
             \end{array}
           \right),
Z_{2}=\left(
             \begin{array}{ccc}
               0 & 0 & 1 \\
               0 &  0 & 0 \\
               -1 & 0 & 0 \\
             \end{array}
           \right),Z_{3}=\left(
             \begin{array}{ccc}
               0 & -1 & 0 \\
               1 &  0 & 0 \\
               0 & 0 & 0 \\
             \end{array}
           \right)$$. The general form of linear map is given by
\begin{eqnarray*}\Theta_{\beta}:\mathfrak{so}(3)&\longrightarrow&
\mathfrak{so}(3): X\mapsto\lambda X,\; with\; X=k Z_{3},\;k\in
\mathrm{I\! R}.
\end{eqnarray*}
We have $\beta=a Z_{3}$ and $\eta=\frac{1}{a}Z_{3}$. Find the value
of $\lambda$ such that $\Theta_{\beta}(\beta)=\eta$. We have
\begin{equation*}
   \Theta_{\beta}(aZ_{3})=\lambda aZ_{3}=\frac{1}{a}Z_{3}.
\end{equation*}
So, $\lambda=\frac{1}{a^{2}}$. We obtain the following application
\begin{eqnarray*}\Theta_{\beta}:\mathfrak{so}(3)&\longrightarrow&
\mathfrak{so}(3): X\mapsto\frac{1}{a^{2}}X.
\end{eqnarray*}
$\Theta_{\beta}$ is linear map for fixed  beta, because for all $X,Y
\in \mathfrak{so}(3)$, and for all $\alpha_{1},\alpha_{2}\in
\mathrm{I\! R}$ we have:
\begin{equation*}
    \Theta_{\beta}(\alpha_{1}X+\alpha_{2}Y)=
    \alpha_{1}\Theta_{\beta}(X)+\alpha_{2}\Theta_{\beta}(Y)
\end{equation*}
 The dual potential function is
determined using Legendre's equation
\begin{eqnarray*}\Psi&=&\langle \eta, \beta\rangle-\Phi
\end{eqnarray*}
we obtain
\begin{eqnarray*}\Psi&=&1-\log(2a)
\end{eqnarray*}
So, the distinguished density function is therefore
\begin{equation*}p(\beta,\xi)=\frac{e^{-\langle\beta,\xi\rangle}}{\int^{+\infty}_{0}
e^{-\langle\beta,\xi\rangle}
d\xi}=\frac{e^{-2ax}}{\int^{+\infty}_{0} e^{-2ax} dx}\end{equation*}
we have,
\begin{eqnarray*}\int_{\Omega^{*}} p(\beta,\xi)d\xi&=&\int^{+\infty}_{0}\frac{e^{-\langle\beta,\xi\rangle}}{\int^{+\infty}_{0}
e^{-\langle\beta,\xi\rangle} d\xi}d\xi\\
&&= \frac{1}{2a}\int^{+\infty}_{0}
e^{-\langle\beta,\xi\rangle} dx\\
&=&\frac{\frac{1}{2a}}{\frac{1}{2a}}\\
&=&1\\
 &=&\frac{\int^{+\infty}_{0}
e^{-\langle\beta,\xi\rangle}d\xi}{\int^{+\infty}_{0}
e^{-\langle\beta,\xi\rangle} d\xi}
\end{eqnarray*} Furthermore,

\begin{eqnarray*}\mathrm{I\! E}_{\beta}[\xi]&=&\int^{+\infty}_{0}\xi\frac{e^{-\langle\beta,\xi\rangle}}{\int^{+\infty}_{0}
e^{-\langle\beta,\xi\rangle} d\xi}d\xi\end{eqnarray*} By setting
$\chi(\beta)=\int^{+\infty}_{0} e^{-2ax} dx=\frac{1}{2a}$ we obtain

\begin{eqnarray*}
\mathrm{I\! E}_{\beta}[\xi]&=&\frac{1}{\chi(\beta)}\int^{+\infty}_{0}\xi e^{-\langle\beta,\xi\rangle}d\xi\end{eqnarray*}

\begin{eqnarray*}\frac{\partial \chi(\beta)}{\partial \beta}&=&\frac{\partial }{\partial \beta}\int^{+\infty}_{0}
e^{-\langle\beta,\xi\rangle}d\xi\end{eqnarray*} Under regularity
conditions, we have
\begin{eqnarray*}\frac{\partial }{\partial \beta}\int^{+\infty}_{0}
e^{-\langle\beta,\xi\rangle}d\xi&=&\int^{+\infty}_{0}\frac{\partial
}{\partial \beta} e^{-\langle\beta,\xi\rangle}d\xi\end{eqnarray*}
So,
\begin{eqnarray*}\frac{\partial \chi(\beta)}{\partial \beta}&=&\int^{+\infty}_{0}\frac{\partial
}{\partial \beta}
e^{-\langle\beta,\xi\rangle}d\xi=-\int^{+\infty}_{0}\xi
e^{-\langle\beta,\xi\rangle}d\xi\end{eqnarray*} So,
\begin{eqnarray*}\mathrm{I\! E}_{\beta}[\xi]&=&-\frac{1}{\chi(\beta)}\frac{\partial \chi(\beta)}{\partial \beta}\\
&=&-\frac{\partial \log \chi(\beta)}{\partial
\beta}\\
&=&\frac{\partial \Phi(\beta)}{\partial \beta}\\
&=&\eta(\beta)\\
&=&\left(
             \begin{array}{ccc}
               0 & -\frac{1}{a^{2}} & 0 \\
                \frac{1}{a^{2}}& 0 & 0 \\
               0 & 0 & 0 \\
             \end{array}
           \right)
\end{eqnarray*}

\begin{eqnarray*}\mathrm{I\! E}_{\beta}[\xi]&=&\int^{+\infty}_{0}\xi\frac{e^{-\langle\beta,\xi\rangle}}{\int^{+\infty}_{0}
e^{-\langle\beta,\xi\rangle} d\xi}d\xi= \frac{\int^{+\infty}_{0}\xi
e^{-\langle\beta,\xi\rangle}d\xi}{\int^{+\infty}_{0}
e^{-\langle\beta,\xi\rangle} d\xi}=\eta(\beta).\end{eqnarray*}
\end{proof}

\begin{theorem}\label{th9} Let $\left\{Z_{1},Z_{2},Z_{3}\right\}$ be a basis on
$\mathfrak{so}(3)$ with $Z_{1}=\left(
             \begin{array}{ccc}
               0 & 0 & 0 \\
               0 &  0 & -1 \\
               0 & 1 & 0 \\
             \end{array}
           \right),
Z_{2}=\left(
             \begin{array}{ccc}
               0 & 0 & 1 \\
               0 &  0 & 0 \\
               -1 & 0 & 0 \\
             \end{array}
           \right),Z_{3}=\left(
             \begin{array}{ccc}
               0 & -1 & 0 \\
               1 &  0 & 0 \\
               0 & 0 & 0 \\
             \end{array}
           \right)$. Let $\mathfrak{so}(3)=span\{Z_{3}\}\oplus
span\{Z_{1},Z_{2}\}$ with $span\{Z_{1},Z_{2}\}=\mathfrak{p}$ be the
Cartan decomposition associated with the symmetric space
$SO(3)\setminus SO(2)=S^{2}$. Let $\beta=\left(
             \begin{array}{ccc}
               0 & -a & 0 \\
               a & 0 & 0 \\
               0 & 0 & 0 \\
             \end{array}
           \right)
\in span\{Z_{3}\}$, consider a linear map
$\Theta:\mathfrak{so}(3)\longrightarrow \mathfrak{so}(3)$ be an
Ad-invariant linear map on $span\{Z_{3}\}$, with $\Theta=Ad_{J}$,
where $J=\left(
             \begin{array}{ccc}
               1 & 0 & 0 \\
               0 & 1 & 0 \\
               0 & 0 & -1 \\
             \end{array}
           \right)\in O(3)$. Then the generalized Fisher
information metric induced on the tangent space of the orbit
$\mathfrak{p}$ with the normalization chosen is given by
\begin{equation*}
    I(\beta)=\left(
               \begin{array}{cc}
                 2a^{2} & 0 \\
                 0 & 2a^{2} \\
               \end{array}
             \right).
\end{equation*}
and the generalized Fisher information metric on Lie group $SO(3)$
on the basis $\left\{Z_{1},Z_{2},Z_{3}\right\}$ is given by
\begin{equation*}
    I(\beta)=\left(
             \begin{array}{ccc}
               2a^{2} & 0 & 0 \\
               0 &  2a^{2} & 0 \\
               0 & 0 & 0 \\
             \end{array}
           \right)
\end{equation*}

\end{theorem}
\begin{proof}
Let\\ $SO(3)=\left\{a_{1}(t)=\left(
             \begin{array}{ccc}
               1 & 0 & 0 \\
               0 &  \cos(t) & -\sin(t) \\
               0 & \sin(t) & \cos(t) \\
             \end{array}
           \right),a_{2}(t)=\left(
             \begin{array}{ccc}
               \cos(t) & 0 & \sin(t) \\
               0 &  1 & 0 \\
               -\sin(t) & 0 & \cos(t) \\
             \end{array}
           \right),a_{3}(t)=\left(
             \begin{array}{ccc}
               \cos(t) & -\sin(t) & 0 \\
               \sin(t) &   \cos(t) & 0 \\
               0 & 0 & 1 \\
             \end{array}
           \right)\right\}$
we have $Z_{1}=\frac{d}{dt}a_{1}(t)\mid_{t=0},
Z_{2}=\frac{d}{dt}a_{2}(t)\mid_{t=0},Z_{3}=\frac{d}{dt}a_{3}(t)\mid_{t=0}$
and we obtain $Z_{1}=\left(
             \begin{array}{ccc}
               0 & 0 & 0 \\
               0 &  0 & -1 \\
               0 & 1 & 0 \\
             \end{array}
           \right),
Z_{2}=\left(
             \begin{array}{ccc}
               0 & 0 & 1 \\
               0 &  0 & 0 \\
               -1 & 0 & 0 \\
             \end{array}
           \right),Z_{3}=\left(
             \begin{array}{ccc}
               0 & -1 & 0 \\
               1 &  0 & 0 \\
               0 & 0 & 0 \\
             \end{array}
           \right).$

           \begin{eqnarray*}
           \left[Z_{3},Z_{3}\right]&=& 0\\
           \left[Z_{3},Z_{1}\right]&=& Z_{2} \in\mathfrak{p}\\
           \left[Z_{1},Z_{2}\right]&=& Z_{3}\in span\{Z_{3}\} \\
           \left[Z_{2},Z_{1}\right]&=& Z_{3} \\
             \left[\beta,Z_{1}\right]&=& aZ_{2} \\
             \left[\beta,Z_{2}\right]&=& aZ_{1} \\
             \left [\beta,Z_{3}\right]&=& 0
              \end{eqnarray*}

             \begin{eqnarray*}\left[Z_{1},\left [[\beta,Z_{1}\right]\right] &=& a Z_{3}\\
            \left[Z_{1},\left[\beta,Z_{2}\right]\right] &=& 0\\
             \left[Z_{2},\left[\beta,Z_{2}\right]\right] &=& -a Z_{3}\\
             \left[Z_{2},\left[\beta,Z_{1}\right]\right] &=& 0\\
             \left[Z_{1},\left[\beta,Z_{3}\right]\right] &=& 0\\
             \left [Z_{2},\left[\beta,Z_{3}\right]\right] &=& 0\\
               \left[Z_{3},\left[\beta,Z_{3}\right]\right] &=& 0\\
                \left[Z_{3},\left[\beta,Z_{2}\right]\right] &=& 0\\
                 \left[Z_{3},[\beta,Z_{1}]\right] &=& 0
                 \end{eqnarray*}
                 \begin{eqnarray*}
                \Theta(Z_{3})&=&Ad_{J}( Z_{3})=Z_{3}\\
                \Theta(Z_{1})&=&Ad_{J}( Z_{1})= -Z_{1}\\
                \Theta(Z_{2})&=&Ad_{J}( Z_{2})= -Z_{2}\\
                \tilde{\Theta}(Z_{1},Z_{3}) &=&-\langle Z_{1},Z_{3}\rangle=0\\
                \tilde{\Theta}(Z_{2},Z_{3}) &=&-\langle Z_{2},Z_{3}\rangle=0\\
                \tilde{\Theta}(Z_{2},Z_{2}) &=&-\langle Z_{2},Z_{2}\rangle=0\\
                \tilde{\Theta}(Z_{1},Z_{1}) &=&-\langle Z_{1},Z_{1}\rangle=0\\
                \tilde{\Theta}(Z_{3},Z_{1}) &=&\langle Z_{3},Z_{1}\rangle=0\\
                \tilde{\Theta}(Z_{2},Z_{2}) &=&\langle
                Z_{3},Z_{2}\rangle=0.
           \end{eqnarray*}
 \begin{eqnarray*}
         g_{\beta}\left(\left[\beta,Z_{1}\right],\left[\beta,z_{2}\right]\right) &=&
         \tilde{\Theta}\left(Z_{1},\left[\beta,Z_{2}\right]\right)+\langle\eta, \left[z_{1},\left[\beta,Z_{2}\right]\right] \\
        &=& a\tilde{\Theta}\left(Z_{1},Z_{3}\right)+\langle\eta, \left[Z_{1},\left[\beta,Z_{2}\right]\right] \\
         &=&-2a^{2}\\
         g_{\beta}\left(\left[\beta,Z_{2}\right],\left[\beta,Z_{1}\right]\right) &=&
         \tilde{\Theta}\left(Z_{2},\left[\beta,Z_{1}\right]\right)+\langle\eta, \left[Z_{2},\left[\beta,Z_{1}\right]\right] \\
         &=& a\tilde{\Theta}\left(Z_{1},Z_{2}\right)+\langle\eta, \left[Z_{1},\left[\beta,Z_{2}\right]\right] \\
         &=&0\\
          g_{\beta}\left(\left[\beta,Z_{1}\right],\left[\beta,Z_{2}\right]\right) &=&
         \tilde{\Theta}\left(Z_{1},\left[\beta,Z_{2}\right]\right)+\langle\eta, \left[Z_{1},\left[\beta,Z_{2}\right]\right] \\
         &=&0\\
 g_{\beta}\left(\left[\beta,Z_{2}\right],\left[\beta,Z_{2}\right]\right) &=&
         \tilde{\Theta}\left(Z_{2},\left[\beta,Z_{2}\right]\right)+\langle\eta, \left[Z_{2},\left[\beta,Z_{2}\right]\right] \\
         &=&2a^{2}\\
         g_{\beta}\left(\left[\beta,Z_{3}\right],\left[\beta,z_{2}\right]\right) &=&
         -a\tilde{\Theta}\left(Z_{3},\beta,Z_{1}\right)+\langle\eta, \left[Z_{3},\left[\beta,z_{2}\right]\right] \\
         &=&0\\
         g_{\beta}\left(\left[\beta,Z_{3}\right],\left[\beta,Z_{1}\right]\right) &=&
         a\tilde{\Theta}\left(Z_{3},\beta,Z_{2}\right)+\langle\eta, \left[Z_{3},\left[\beta,z_{1}\right]\right] \\
         &=&0\\
         g_{\beta}\left(\left[\beta,Z_{1}\right],\left[\beta,Z_{3}\right]\right) &=&0\\
         g_{\beta}\left(\left[\beta,Z_{2}\right],\left[\beta,Z_{3}\right]\right) &=&0\\
         g_{\beta}\left(\left[\beta,Z_{3}\right],\left[\beta,Z_{3}\right]\right)
         &=&0.
\end{eqnarray*}

\end{proof}

\begin{theorem}\label{th10}
Let $\beta=\left(
             \begin{array}{ccc}
               0 & -a & 0 \\
               a & 0 & 0 \\
               0 & 0 & 0 \\
             \end{array}
           \right)\in \mathfrak{so}(3), \; a\in \mathrm{I\!
R}$.  The gradient system is given by
\begin{equation}\label{ha1}
\dot{\beta}=-\beta\end{equation} where its reduced form is given by
\begin{equation}\label{haa} \dot{a}=-a.\end{equation}
Consider the subspace $\Omega^{*}=\left\{\xi=\left(
             \begin{array}{ccc}
               0 & -x & 0 \\
               x & 0 & 0 \\
               0 & 0 & 0 \\
             \end{array}
           \right)
                                , x >0 \right\}\subset\mathfrak{so}(3)$, and $\nu^{*}:O(3)\times
\Omega^{*}\longrightarrow\Omega^{*}:(J,\xi)\longmapsto
Ad^{*}_{J}(\xi)$ the action of Lie group $O(3)$. Then there exists a
foliation whose leaves are the orbit of the $1$-dimensional foliated
action $\mathcal{O}_{\xi}(\nu^{*})=\left\{\left(
             \begin{array}{ccc}
               0 & -x & 0 \\
               x & 0 & 0 \\
               0 & 0 & 0 \\
             \end{array}
           \right),\;
x>0\right\}$ diffeomorphic to the $1$-dimensional torus. In this
case the system (\ref{ha1}) and (\ref{haa}) are completely
integrable.
\end{theorem}

\begin{proof}
Let  $\nu^{*}:O(3)\times
\Omega^{*}\longrightarrow\Omega^{*}:(g,\xi)\longmapsto
Ad^{*}_{g}(\xi)$, us first show that $\nu^{*}$ is action.
 Let
                                 $O(3)=\left\{\left(
             \begin{array}{ccc}
               1 & 0 & 0 \\
               0 & 1 & 0 \\
               0 & 0 & -1 \\
             \end{array}
           \right),\left(
             \begin{array}{ccc}
               1 & 0 & 0 \\
               0 & -1 & 0 \\
               0 & 0 & 1 \\
             \end{array}
           \right),\left(
             \begin{array}{ccc}
               -1 & 0 & 0 \\
               0 & 1 & 0 \\
               0 & 0 & 1 \\
             \end{array}
           \right)\right\}$ the Lie group.  For any $ J\in O(3)$, show that
\begin{equation*}\nu^{*}(J,\xi)=\xi.\end{equation*} We have
\begin{equation*}\nu^{*}(J,\xi)=Ad^{*}_{J}(\xi)=\left(
             \begin{array}{ccc}
               1 & 0 & 0 \\
               0 & 1 & 0 \\
               0 & 0 & -1 \\
             \end{array}
           \right)\left(
             \begin{array}{ccc}
               0 & -x & 0 \\
               x & 0 & 0 \\
               0 & 0 & 0 \\
             \end{array}
           \right)\left(
             \begin{array}{ccc}
               1 & 0 & 0 \\
               0 & 1 & 0 \\
               0 & 0 & -1 \\
             \end{array}
           \right)=\left(
             \begin{array}{ccc}
               0 & -x & 0 \\
               x & 0 & 0 \\
               0 & 0 & 0 \\
             \end{array}
           \right)\end{equation*}
     Let $J_{1}=\left(
             \begin{array}{ccc}
               1 & 0 & 0 \\
               0 & 1 & 0 \\
               0 & 0 & -1 \\
             \end{array}
           \right),\;J_{2}=\left(
             \begin{array}{ccc}
               1 & 0 & 0 \\
               0 & -1 & 0 \\
               0 & 0 & 1 \\
             \end{array}
           \right)$,\; show that
                                 $\nu^{*}\left(J_{1}J_{2},\xi\right)=\nu^{*}\left(J_{1},\nu^{*}\left(J_{2},\xi\right)\right).$
We have

\begin{equation*}
    \nu^{*}\left(J_{1}J_{2},\xi\right)=\left(J_{1}J_{2}\right)^{-1}\xi\left(J_{1}J_{2}\right)\end{equation*}.
So,
\begin{equation}\label{c0}\nu^{*}\left(J_{1}J_{2},\xi\right)=J_{2}^{-1}\left(J_{1}^{-1}\xi
J_{1}\right)J_{2}=J_{2}^{-1}\xi J_{2}=\xi.
\end{equation}
 In same, we have
\begin{equation}\label{c01}\nu^{*}\left(g_{1},\nu^{*}\left(J_{2},\xi\right)\right)=\nu^{*}\left(J_{1},\xi\right)=\xi.\end{equation}
Using (\ref{c}) and (\ref{c1}) we obtain
$\nu^{*}\left(J_{1}J_{2},\xi\right)=\nu^{*}\left(J_{1},\nu^{*}\left(J_{2},\xi\right)\right)$.
The defined gradient system will be given by
\begin{equation}
\dot{\beta}=-a^{2}\partial_{\beta}\Phi(\beta)\rangle\end{equation}
where $\partial_{\beta}:=\frac{\partial}{\partial \beta}$. So,
\begin{equation}\label{a}
\dot{\beta}=-\beta=\left(
             \begin{array}{ccc}
               0 & a & 0 \\
               -a & 0 & 0 \\
               0 & 0 & 0 \\
             \end{array}
           \right)\end{equation}
In the same,
\begin{equation*}
\dot{\beta}=\frac{\partial \beta}{\partial a}\dot{a}\end{equation*}
We obtain  \begin{equation} \dot{\beta}=\left(
             \begin{array}{ccc}
               0 & -\dot{a} & 0 \\
               \dot{a} & 0 & 0 \\
               0 & 0 & 0 \\
             \end{array}
           \right)\end{equation} Thus by
identification with \ref{a}, we have \begin{equation*}
\dot{a}=-a\end{equation*}
\end{proof}

\vfill\eject

\section{Conclusion}\label{sec6}
This work has generalized the Fisher metric in the framework of the
thermodynamics of Souriau Lie groups, highlighting the deep
connections between information geometry,  and algebraic structures.
Focusing on the $SO(2)$ and $SO(3)$ Lie groups, we have explicitly
characterized the effect of central 2-cocycles on the Fisher metric
and demonstrated the full integrability of the associated gradient
systems.  These results not only enrich existing theory but also
open the way to promising applications in fields such as geometric
machine learning and neuro-inspired optimization.

%Ce travail a permis de généraliser la métrique de Fisher dans le
%cadre de la thermodynamique des groupes de Lie de Souriau, en
%mettant en lumière les liens profonds entre géométrie de
%l'information, systèmes hamiltoniens et structures algébriques. En
%nous concentrant sur les groupes de Lie $ SO(2)$ et $ SO(3)$, nous
%avons explicitement caractérisé l'effet des 2-cocycles centraux sur
%la métrique de Fisher et démontré l'intégrabilité complète des
%systèmes gradients associés. La construction de représentations de
%Lax pour ces systèmes, ainsi que la description des orbites
%coadjointes et des structures symplectiques sous-jacentes,
%confirment la pertinence de l'approche géométrique pour l'étude des
%modèles statistiques invariants. Ces résultats ne seulement
%enrichissent la théorie existante, mais ouvrent également la voie à
%des applications prometteuses dans des domaines tels que
%l'apprentissage automatique géométrique et l'optimisation
%neuro-inspirée. Des travaux futurs pourraient étendre cette approche
%à d'autres groupes de Lie, explorer les implications pour la
%quantification de l'information en présence de symétries, ou
%approfondir les connexions avec la théorie des systèmes dynamiques
%et la physique théorique.

\backmatter

\bmhead{Supplementary information} This manuscript has no additional
data.

\bmhead{Acknowledgments} We thank all the members of the Algebra,
Geometry and Applications Laboratory  of the University of Yaounde1
for their suggestions in the work. We thank Professor Joseph Dongho
of the University of Maroua for his comments without forgetting
Professor Thomas Bouetou Bouetou of the Polytechnic School of
Yaounde1. We thank Professor Michel Nguiffo Boyom of the Montpellier
Institute Alexander Grothendiek for his relevant comments and
questions in order to improve the work. We thank Professor
Fr$\acute{e}$d$\acute{e}$ric Barbresco of Thalesgroup, for these
lines of research. A big thank you to Frankel Nozah Mba and Romain
Boris Bopda Gopdjum for their comments and work provided in order to
improve the paper.
\section*{Declarations}
This article has no conflict of interest to the journal. No
financing with a third party.
\begin{itemize}
\item No Funding
\item No Conflict of interest/Competing interests (check journal-specific guidelines for which heading to use)
\item  Ethics approval
\item  Consent to participate
\item  Consent for publication
\item  Availability of data and materials
\item  Code availability
\item Authors' contributions
\end{itemize}

\noindent
%%If any of the sections are not relevant to your manuscript, please include the heading and write `Not applicable' for that section.

%%===================================================%%
%% For presentation purpose, we have included        %%
%% \bigskip command. please ignore this.             %%
%%===================================================%%
%%\bigskip
%%\begin{flushleft}%
%%Editorial Policies for:

%%\bigskip\noindent
%%%%Springer journals and proceedings: \url{https://www.springer.com/gp/editorial-policies}

\bigskip\noindent

\bibliography{sn-bibliography}% common bib file
%% if required, the content of .bbl file can be included here once bbl is generated
%%\input sn-article.bbl

\end{document}